\documentclass[reqno]{amsart}

\usepackage[utf8]{inputenc}

\usepackage{fullpage}

\usepackage{booktabs}
\usepackage{array}

\usepackage{paralist}
\usepackage{thmtools}

\newcommand{\aut}{\mathfrak{aut}}
\newcommand{\hol}{\mathfrak{hol}}

\newcommand{\C}{\mathbb{C}}

\newcommand{\R}{\mathbb{R}}

\newcommand{\N}{\mathbb{N}}

\newcommand{\dop}[1]{\frac{\partial}{\partial #1}}

\newcommand{\todo}[1]{}

\DeclareMathOperator{\ord}{ord}

\DeclareMathOperator{\imag}{Im}
\DeclareMathOperator{\real}{Re}

\newtheorem{thm}{Theorem}
\newtheorem{theorem}{Theorem}

\newtheorem{lemma}[thm]{Lemma}

\newtheorem{proposition}[thm]{Proposition}

\theoremstyle{definition}

\newtheorem{rem}{Remark}
\newtheorem{remark}[rem]{Remark}

\renewcommand{\Re}{\mathop{\rm Re}\nolimits}
\renewcommand{\Im}{\mathop{\rm Im}\nolimits}
\newcommand{\im}{\ensuremath{\mbox{\rm Im}\,}}
\newcommand{\re}{\ensuremath{\mbox{\rm Re}\,}}

\def\beq{\begin{equation}}
\def\eeq{\end{equation}}

\newcommand{\CC}[1]{\mathbb{C}^{#1}}

\newcommand{\RR}[1]{\mathbb{R}^{#1}}
\newcommand{\dz}{\frac{\partial}{\partial z}}
\newcommand{\dw}{\frac{\partial}{\partial w}}

\newcommand{\dd}{\partial}

\numberwithin{equation}{section}

\usepackage{multirow}

\title{Holomorphic vector fields with real integral manifolds}
\author{Martin Kol\'a\v{r}}
\address{Department of Mathematics and Statistics, Masaryk University, Brno, Czech Republic}
\email{mkolar@math.muni.cz}
\author{Ilya Kossovskiy}
\address{Department of Mathematics, Sustech University, Shenzhen, China $\&$ Department of Mathematics and Statistics, Masaryk University, Brno, Czech Republic $\&$ TU Wien, Vienna, Austria}
\email{ilyakos@sustech.edu.cn}
\author{Bernhard Lamel}
\address{Faculty of Mathematics, University of Vienna, Austria}
\email{bernhard.lamel@univie.ac.at}

\usepackage{hyperref}
\usepackage{bigdelim}
\begin{document}

\begin{abstract}
  We classify singular holomorphic vector fields in $\CC{2}$ admitting
  a (Levi-nonflat) real-analytic invariant $3$-fold through the
  singularity. In this way, we complete the classification of
  infinitesimal symmetries of real-analytic Levi-nonflat hypersurfaces
  in complex two-space. The classification of holomorphic vector fields obtained in the paper has very interesting overlaps with the recent Lombardi-Stolovitch classification theory for holomorphic vector fields at a singularity. In particular, we show that most of the resonances arising in Lombardi-Stolovitch theory do {\em not} occur under the presence of (Levi-nonflat) integral manifolds.   
\end{abstract}

\maketitle

\section{Introduction}

In this paper, we study a new intriguing phenomenon at the
intersection of Complex Dynamics and CR geometry by considering real
invariant objects of complex vector fields.  From the CR geometry
point of view, we complete the classification of possible
infinitesimal symmetries of germs of real hypersurfaces in
$\mathbb C^2$, by considering the remaining class of nonminimal
(infinite type) manifolds. In particular, our results imply the
classification of pairs $(M,Y)$, where $M$ is a real-analytic
Levi-nonflat hypersurface and $Y$ its infinitesimal symmetry, i.e. a
holomorphic vector field whose real part is tangent to $M$.

From the Complex Dynamics perspective, we classify singular
holomorphic vector fields $X$, which admit a Levi-nonflat
$3$-dimensional {\em real (analytic) integral manifold} through the
singularity. By a real integral manifold we mean a real submanifold
$M \subset \CC{N}$ which is invariant under the {\em real} flow of
$X$, or equivalently that $\re X$ is tangent to $M$.  It turns out
that the existence of such a manifold has very strong consequences for
the presence of resonances in the vector field and the possibility to
build an {\em analytic} transformation bringing the vector field to
its Lombardi-Stolovitch normal form \cite{LS} (which is a recent natural development of the Poincar\'e-Dulac normal form to the case of higher order degeneracies of a vector field).  This exhibits new and
very interesting phenomena. In particular, we show that holomorphic
vector field in $\CC{2}$ with a real integral manifold admit only
finitely many resonant terms in their normal form (actually, at most
one such term besides one exceptional type).

On the CR geometry side, it has been known since the work of Poincar\'e
\cite{poincare} that symmetries of real hypersurfraces in complex
space are rare, i.e. they are uniquely determined by their finite jet
at a point, under mild nondegeneracy assumptions. Turning to the
$\CC{2}$ case, the classification of possible symmetries of
real-analytic hypersurfaces in the {\em Levi-nondegenerate} case
follows from the Cartan's theory \cite{cartan}, later generalized to
Levi-nondegenerate hypersurfaces of arbitrary dimensions by Tanaka and
Chern and Moser \cite{tanaka,chern}. A detailed description was
obtained in the work \cite{belold} of Beloshapka (see also
Kruzhilin-Loboda \cite{kruzh} for the higher dimemsional case). It
follows from this work that either every infinitesimal symmetry is
linearizable, or the hypersurface is {\em spherical} (and thus the
vector field is biholomorphic to a quadratic vector field from the
projective algebra $\mathfrak{pgl}(2,\CC{})$). A normal form in this
case is a linear combination of the vector fields {\bf 1} to {\bf 4}
below. In case a hypersurface is Levi-degenerate at a reference point
$p$ but is {\em of finite type} (i.e., there is no complex disc
through $p$ lying in $M$), the classification of symmetries was
completed in the work \cite{kolar} of the first author. The normal
form of $X$ is then given by a linear combination of the vector
fields {\bf 4} to {\bf 6} below. For
developments in higher dimension, see the work of the first author
with Meylan and Zaitsev \cite{kmz} and subsequent publications.

In the opposite extreme case, when the hypersurface is {\em
  Levi-flat}, the Lie algebra of infinitesimal CR automorphisms of $M$
is infinite dimensional, since the manifold is locally biholomorphic
to a flat plane. In fact, Levi-flat hypersurfaces exhibit fascinating
{\em global} dynamics (see, e.g., Ilyashenko-Yakovenko
\cite{ilyashenko} and references therein, and also Lins-Neto
\cite{LN}, Brunella \cite{brunella}, and subsequent publications), but
nothing particular can be said in terms of local
behaviour. Interesting phenomena though occur if one allows {\em
  singularities} of hypersurfaces (see e.g. Burns-Gong \cite{BG}), but
these considerations are outside of the scope of the present paper.

The intermediate case of {\em infinite type} analytic hypersurfaces
(i.e. those which admit a holomorphic disc through $p$ but at the same time
are Levi-nonflat) turns out to be much more interesting and difficult at
the same time. The (locally unique)
holomorphic disc through $p$ surves as a {\em separatrix}
for a vector field $X$ for which $M$ is a real integral
manifold. Infinite type hypersurfaces and their symmetries have been
studied intensively in recent years: see e.g. Kowalski
\cite{kowalski}, Zaitsev \cite{zaitsev}, Ebenfelt, Zaitsev, and the
third author \cite{ELZ}, first and third author \cite{KL}, Shafikov
and the second author \cite{nonminimal,divergence,nonminimalODE},
second and third author \cite{nonanalytic}, Ebenfelt with the second
and third authors \cite{EKL}. The difficulties here are illuminated by
the fact that, in contract with the finite type case, an infinitesimal
authomorphism is {\em not} anymore determined uniquely by any fixed
finite jet (see \cite{kowalski,zaitsev}), i.e., for any
$l\in\mathbb N$ there is a pair $(M,X)$ with $\re X$ tangent to $M$,
$j^l_pM=0$, but $X\not\equiv 0$. The respective examples are
discovered as certain {\em sphere blow-ups}.

We point out that invariant submanifolds of the lower dimensions $1$ and $2$ also naturally occur for holomorphic vector fields in $\CC{2}$. The classification of $1$-dimensional analytic integral manifolds is trivial though (all of them can be locally flattened), in the same way as $2$-dimensional ones at their {\em totally real} points. At {\em CR singular} points, however, the classification is highly nontrivial (e.g. Moser-Webster \cite{mw}, Gong \cite{gong}, Huang-Yin \cite{hy}). Describing holomorphic vector fields with CR singular integral manifolds is a very interesting problem which is, however, outside of the scope of this paper.

Let us now describe in more detail the Complex Dynamics point of
view. Consider a holomorphic vector field of the form
$$ X =   \sum_{j=1}^n f_j \dd_{z_j}, $$
satisfying $X(0) = 0$.  The main aim is to understand the local
dynamics of such a vector field in a neighbourhood of the singularity
at the origin.  Poincar\' e, inspired by problems of celestial
mechanics, initiated the classical classification theory at a
singularity for such vector fields in the late 19-th century (which we
now refer to as the {\em Poincar\'e-Dulac theory}).  He realized that on
one hand, in general it is impossible to find the phase portrait,
i.e. solve the corresponding ODE system $\dot Z = X(Z)$. On the other
hand, it may be possible to find a much simpler {\em normal form} for
the system, which can be integrated and used to investigate local
dynamics.
In the situation of a nondegenerate linearization of the system, we
wish to transform the system into a linear one.  However, as it turns
out, there can be obstructions for doing so. In fact, the obstructions
are of two kinds.  First, there can be {\em resonances} which prevent
the field to be linearized. Second, one may encounter {\em small
  divisors}, which prevents convergence of a formal power series
normalization. In the situation of a degenerate linearization, further
obstructions occur (leading to the concept of {\em Stokes data} and
subsequently to various {\em Stokes phenpomena}). For more on this, we
refer to the excellent book of Ilyashenko-Yakovenko \cite{ilyashenko} on the subject,
and the recent papers \cite{LS} of Lombardi-Stolovitch and Stolovitch \cite{bigdenom} containing
important developments of the Poincar\'e-Dulac theory for the situation
of vector fields with a degenerate linearization.

An important tool for understanding the local dynamics of a vector
field are {\em invariant objects}.  Let us illustrate this with some
examples:
\begin{compactenum}
\item By a celebrated result of Camacho-Sad \cite{camacho}, in complex dimension $n=2$, there
  always exists a separatrix, i.e.  an integral complex analytic set
  passing through zero. As long as it is smooth, we obtain an
  invariant manifold.

\item Another example of an integral manifold is a real Levi-flat
  hypersurface, which can arise, e.g.  from the separatrix and
  integral curves passing through neighbouring regular points.

\item The existence of invariant Levi flat sets can already have
  rather strong consequences (for an example, see e.g. Burns-Gong
  \cite{BG}).

\item The existence of an invariant manifold may lead to the
  convergence of a normal form for a Hamiltonian system, see e.g. Gong
  \cite{gong}.

\end{compactenum}

The main discovery of the present paper is that the existence of a
Levi-nonflat analytic invariant manifold for a holomorphic vector
field $X$ in $\CC{2}$ through a point $p$ implies the existence of an
unexpectedly simple normal form for $X$ at $p$. In particular,
compared to the general Poincar\'e-Dulac/Lombardi-Stolovitch theory,
very few resonances show up (in fact, at most one resonant term in the
normal form is possible, besides one exceptional case), and (with the
same exception) the normalization appears to be necessarily {\em
  convergent}, in contrast with the general theory. A detailed list of
normal forms is given below.

\bigskip

\begin{center}{ \bf A. Finite type case}
\end{center}
\begin{equation*}
  \begin{aligned}
    &{\bf 1.\,\,}  z \dd_z + 2 w \dd_w
    \\
    &{\bf 2.\,\,}  (a w + 2i \bar a z^2)\dd_z + 2i \bar a zw \dd_w,\quad a\in\CC{*}
    \\
    &{\bf 3.\,\,}  zw\dd_z + w^2\dd_w
    \\
    &{\bf 4.\,\,}  i z \dd_z
    \\
    &{\bf 5.\,\,}  z\dd_z +  l w\dd_w, \ \ \ \ \ l\in\mathbb{Z}_{\geq 3}
    \\
    &{\bf 6.\,\,}  \frac1k zw\dd_z + w^2\dd_w, \ \ \ \ k\in\mathbb{Z}_{\geq 2}
  \end{aligned}
\end{equation*}

\bigskip 

\begin{center}{\bf B.  Infinite type case}
\end{center}
\begin{equation*}
  \begin{aligned}
    {\bf 7.\,\,} & w^k \dd_z,\quad k\in\mathbb{N} 
      \hspace{.5cm} \\
    {\bf 8.\,\,} & (w^k + r w^{2k-1}) \dd_w, \ \ \  k\in\mathbb{Z}_{\geq 2},\,r\in\RR{}
      \hspace{.5cm} \\
    {\bf 9.\,\,} & w \dd_w \ \ \ 
      \hspace{.5cm} \\
    {\bf 10.\,\,} &  w^k(\mu z \dd_z + w \dd_w), \quad k\in\mathbb{Z}_{\geq 0}, \  \Re \mu  \in \mathbb Q^-, \  \Im \mu  \neq 0
      \hspace{.5cm} \\
    {\bf 11.\,\,} &  w^k(\mu  z \dd_z + w \dd_w) + \eta w^{2k +1},  \quad k\in\mathbb{N},\, \ 
      \mu  \in \mathbb Q^-, \ \eta \in \mathbb R
      \hspace{.5cm} \\
    {\bf 12.\,\,} & \mu  z \dd_z + w \dd_w, \quad \mu  \in \mathbb Q^-
      \hspace{.5cm} \\
    {\bf 13.\,\,} &  i z \dd_z
    \\
    {\bf 14.\,\,} &  i z w^{k}(1+ c_1w+...+ c_{q}w^{q})  \dd_z 
      +  (r w^{k+1+q}+tw^{2q+2k+1}) \dd_w,  \,\,\mbox{where}
    \\& k\in\mathbb{Z}_{\geq 0}, \, q\in\mathbb{N}, \, r \in \mathbb R^*, \, t, c_j \in \mathbb R 
    \hspace{.5cm} \\
  \end{aligned}
\end{equation*}

Now our main result is as follows.

 \begin{theorem}\label{main}
   Let $X$ be a nonzero holomorphic vector field which is singular
   at a point $p\in \C^2$, i.e. $X(p)=0$, and admits a germ of a
   Levi-nonflat real-analytic integral $3$-fold $M$ through $p$. Then
   the following dichotomy holds:

   \smallskip

   \noindent (i) if the integral manifold $M$ is of finite type at
   $p$, then $X$ is biholomorphic at $p$ to a (real) linear
   combination of the normal forms {\bf 1 -- 6}.

   \smallskip

   \noindent (ii) if the integral manifold $M$ is of infinite type at
   $p$, then (after a rescaling by $c\in\RR{*}$) the vector field $X$
   is either biholomorphic at $p$ to one of the normal forms {\bf 7 --
     13}, or is formally biholomorphic to the normal form {\bf 14}.

   \smallskip

   Furtheremore, each of the normal forms {\bf 1 -- 14} is realized by
   an actual (analytic) pair $(X,M)$.

 \end{theorem}

\begin{remark}\label{divergent}
  As shown in Section 4 below, the normal form {\bf 14} above is
  actually {\em divergent,} as a rule (in the sense that {\em any}
  normalizing transformation into it is divergent). Such a divergence phenomenon, in a sense, is to be expected due to the divergence results for mappings of infinite type hypersurfaces obtained recently in \cite{divergence,nonanalytic}.  
\end{remark}

The method that we use for the normalization is a far-going
development of the earlier work \cite{KL} of the first and the third
authors (where a complete classification of symmetries was given for
so-called {\em ruled hypersurfaces}), and is based on studying (and
comparing) various weight systems associated with local geometric
invariants of the integral manifold of a vector field. It turns out
that (almost) all possible resonant terms are not compatible with the
obtained weighted decompositions of the manifold and the vector field,
which allows for the desired simple normal form.

\bigskip
\newpage

\begin{center} \Large \bf Acknowledgements          \end{center}

The first and the second author are supported by the GACR grant 22-15012J. The second author is also supported by the FWF grant P34369 and SusTech internal funding Y01286147, and the third author is supported by the FWF grant 10.55776/I4557.

\bigskip

\section{Initial normalization of vector fields tangent to a real
  hypersurface}

Let $$X=P(z,w)\dop{z}+Q(z,w)\dop{w}\not\equiv 0$$ be a holomorphic
vector field in $\CC{2}$ defined in an open neighborhood $U$ of the
origin and vanishing at $0$.  The first key ingredient is to use a
weight system where the weight of $w$ is $1$ and the weight of $z$ is
0, corresponding to the following expansion for $X$:
\begin{equation}\label{key}
  X = \sum_{\ell = k}^\infty X_\ell,
\end{equation} where
$$X_\ell = \alpha_\ell (z) w^\ell \dop{z} + \beta_{\ell}(z) w^{\ell+1} \dop{w},  $$
and $$X_k\not\equiv 0.$$

We first of all observe that the normalization of $X$ is quite simple
in the case when $\alpha_k(0)\neq 0$. Note that $X(0)=0$ implies $k>0$
in the latter case.
\begin{proposition}\label{ord0}
  Assume that in the expansion \eqref{key} one has
  $\alpha_k(0)=\alpha\neq 0$. Then $X$ is biholomorphically equivalent
  to the monomial vector field $$Y:=\alpha w^k\dop{z}$$ near the
  origin.
\end{proposition}
\begin{proof}
  If we try to map the vector field $X$ to the vector field $Y$ by a
  (biholomorphic) transformation
$$z\mapsto f(z,w), \quad w\mapsto wg(z,w),\quad g(0,0)\neq 0, \quad f_z(0,0)\neq 0,$$
then since
$$X=(\alpha w^k+O(w^kz)+O(w^{k+1}))\dop{z}+O(w^{k+1})\dop{w},$$
it is not difficult to compute that $X$ being mapped to $Y$ is
equivalent to:
\begin{equation}\label{CK}
  \alpha g^k=(\alpha+O(f)+O(w))f_z+O(w)f_w, \quad 0=(\alpha+O(f)+O(w))g_z+O(1)(wg_w+g).
\end{equation}
We see now that the Cauchy-Kowalewski theorem is applicable to the PDE
system \eqref{CK} when choosing any analytic initial data
$f(0,w),g(0,w)$ with $f(0,0)=0$, which proves the proposition.
\end{proof}
In view of \autoref{ord0}, we assume from now on that $\alpha_k(0)=0$
in \eqref{key}.
\\
We will now add the condition that {\em $X$ is tangent to a
  real-analytic Levi-nonflat hypersurface $M$ through $0$} - a real
integral manifold of $X$.  As discussed in the Introduction, we assume
$M$ to be {\em of infinite type at $0$}. We also choose local
holomorphic coordinates in which the infinite type locus of the
hypersurface is $E=\{w=0\}$, and the hypersurface itself takes the
form $$M=\{(z,w) \in U \colon \rho(z,\bar z,w,\bar w)=0\}$$ with
$\rho$ satisfying
\begin{equation}\label{mnonminimal}
  \rho(z,\bar z,w,\bar w)=v-u^m\varphi(z,\bar z,u),\quad\varphi(z,\bar z,u)=\varphi_s(z,\bar z)+O(|z|^{s+1})+O(|u|).
\end{equation}
Here we assume that $\varphi_s$ is homogeneous of degree $s$ in $z$, and
$\varphi(z,0,u)=\varphi(0,\bar z, u) = 0$. Such coordinates are often
referred to as {\em normal} (see \cite{ber}), and the latter condition
states that $\varphi$ does not contain any so-called ``harmonic
terms'' (that is, no terms of the form $z^j u^k$ or $\bar z^j
u^k$). The integer $m>0$ is called the {\em nonminimality order}; it
is a biholomorphic invariant (see Meylan \cite[Corollary
2.3]{meylan}). We then make use of the tangency condition (basic
identity) $$\re X(\rho)=0|_{\rho=0}$$ (precisely meaning that $M$ is
an integral manifold of $X$), and obtain:
\begin{equation}\label{basic}
  \begin{aligned}
    \im\bigl(\beta_{k}(z)w^{k+1}+\cdots\bigr)&=mu^{m-1}\varphi(z,\bar z,u)\re(\beta_k(z)w^{k+1})+u^m\re(\varphi_z(z,\bar z,u)
                                               \alpha_k(z)w^{k})+\\
                                             &+u^m\varphi_u(z,\bar z,u)\bigr)\re(\beta_k(z)w^{k+1})+\cdots.
  \end{aligned}
\end{equation}
Here dots stand for terms arising from $X_\ell$ with $\ell>k$, and in
order to have a true equality, we have to substitute
$$w = u + i u^m \varphi (z , \bar z, u).$$ We shall now make some
important conclusions from the basic identity \eqref{basic}, using
extensively the absence of harmonic terms in $\varphi$. As a first
consequence, comparing harmonic terms of the kind
$z^ju^{k+1},\,j\geq 0$, we see that $$\beta_k(z)=B\in\RR{}$$
(i.e. $\beta_k(z)$ is a {\em real constant}).

We now have the following trichotomy (recall that we assume
$\alpha_k(0)=0$):

\begin{compactenum}
\item $\alpha_k \equiv 0$ ;
\item $\alpha_k \not\equiv 0$, $\beta_k \equiv B = 0$;
\item $\alpha_k \not\equiv 0$, $\beta_k \equiv B \neq 0$;
\end{compactenum}

We refer to the first two cases as {\em exceptional} and to the third
one as {\em generic}. Let us make the following simple remarks:

\begin{remark}\label{B=0}
  If $B=\beta_k=0$, then $X_k\not\equiv 0$ implies
  $\alpha_k\not\equiv 0$, i.e. we are in case (2). Then, considering
  in the basic identity terms with $u^{k+m}z^i\bar z^j,\,i,j>0$, we
  see that they can occur only from
  $u^m\re(\varphi_z(z,\bar z,u)\alpha_k(z)w^{k})$. Among the latter
  ones, we consider that with the minimal total degree in $z,\bar z$
  (they must come from $\varphi_s$ and the minimal degree term in
  $\alpha_k(z)$). Then, since their sum must {\em vanish}, it is easy
  to conclude that the only possibility is that \beq\label{ord1}
  \mbox{ord}_0\,\alpha_k=1.  \eeq If we write
  $\alpha_k (z) = Az + \dots $, we see that furthermore $A\in i\RR{}$
  and $\varphi_s=|z|^s$ for an even $s$.
\end{remark}
\begin{remark}
  In the generic case, $B\neq 0$, $\alpha_k\not\equiv 0$, if we again
  consider terms with $u^{k+m}z^i\bar z^j,\, i,j>0$ in the basic
  identity, we see that they arise from
  $u^m\re(\varphi_z(z,\bar z,u)\alpha_k(z)w^{k})$,
  $mu^{m-1}\varphi(z,\bar z,u)\re(\beta_k(z)w^{k+1})$, and
  $\im\bigl(\beta_{k}(z)w^{k+1})$. Since they have to agree with each
  other, it is not difficult to conclude, once again, that
  \eqref{ord1} holds.
\end{remark}

We now proceed with the {\em generic} case, i.e. we assume
$$B\neq 0,\quad \alpha_k(0)=0,\quad A:=\alpha_k'(0)\in \CC{*},$$
and will first give a {\em formal} normalization of the vector field.
For this formal normalization, we do {\em not} assume that $X$ is
tangent to a real hypersurface yet.  Recall that we work with the
grading which associates to $w$ the weight $1$ and to $z$ the weight
$0$, so that for each $\ell\geq 0$ the change of coordinates
\[ \begin{aligned}
  z &= x + y^\ell f(x) \\
  w&= y + y^{\ell+1} g(x),
\end{aligned} \]
is homogeneous of degree $\ell$ in our grading, with the additional understanding that for
$\ell = 0$ we assume that $f = O(x^2)$ and $g = O(x)$. We attempt to remove inductively all terms of homogeneity larger than $k$ (i.e. $X_j$ with $j>k$), and for $j=k$ we attempt to leave only the leading term $w^k(Az\dop{z}+Bw\dop{w})$. Thus, we assume that all $X_j$ with $k\leq j<k+\ell$ are already normalized, and under the above change of coordinates 
the vector field $X$ transforms into:
\[ \begin{aligned}
  \sum_{j<\ell} &\left(   {\alpha_{k+j}}(x) y^{k+j} \dop{x} + \beta_{k+j} y^{{k+j}+1} \dop y  \right) + \\
                &+\left( ({\alpha_{k}'}(x) - \ell B) f(x) -
                  {\alpha_{k}}(x) f'(x) + {k} {\alpha_{k}}(x) g(x)
                  \right) y^{{k}+\ell} \dop{x} \\& + \left(
  ({k}-\ell)B g(x) - {\alpha_{k}}(x) g'(x) \right)y^{{k}+\ell+1}
  \dop{y} + \dots.
\end{aligned} \]
The dots here denote terms of degree exceeding $k+\ell$.

Thus normalizing $X$ to order $k+\ell$ amounts to solving the
following system of Fuchsian (Briott-Bouquet) equations (see
e.g. \cite{laine}):
\begin{equation}
  \label{e:secondsystem} \begin{aligned}
    {A} x f'(x)  - ({A} - \ell B) f(x)  &= O(x^2) f'(x)  + O(x)  f(x) + O(x) g(x) -\alpha_{k+\ell} (x) \\ 
    - B f(x) + {A}x  g'(x)  - ({k}-\ell) B g(x) &=O(x^2) g'(x)  + O(x)  f(x) + O(x) g(x) - \beta_{k+\ell }(x).
  \end{aligned}
\end{equation}
We shall recall the general fact that for a Briott-Bouquet ODE
$$xdY/dx=TY+H(x,Y), \quad T\in\mbox{Mat}(n,\CC{}),\,\,H(0)=0,\,dH(0)=0,$$ 
a formal solution exists if no $n\in\mathbb Z_{\geq 0}$ is an
eigenvalue of the matrix $T$ (and in case $H$ is analytic, the formal
solution is analytic too). If, otherwise, there are nonnegative
integers $n\in\mbox{spec}\,T$, a formal solution exists upon
"correcting" $H(x,Y)$ with a term $cx^n$ (here $c$ is a constant
vector).  It follows that we can remove all powers $x^n$ from
$\alpha_{k+\ell}$ and $\beta_{k + \ell}$ except for those $n$ for
which the matrix
\begin{equation}\label{e:matrix} \begin{pmatrix}
  {A} n - ({A} - \ell B) & 0 \\ B & {A} n - (k -\ell) B
\end{pmatrix}  \end{equation}
fails to be of full rank. If we write $\lambda = \frac{B}{A} $ and set
\[ n_1(\ell) = 1 - \ell \lambda, \quad n_2(\ell) = (k - \ell)
  \lambda, \] then we can normalize to
$\alpha_{k+\ell} (x) = c_\ell x^{n_1 (\ell)}$ and
$\beta_{k+\ell} (x) = d_\ell x^{n_2(\ell)}$ (with the understanding
that $c_\ell = 0$ if $n_1 (\ell)$ is not an integer and $d_\ell = 0$
if $n_2 (\ell)$ is not an integer).  For $\ell = 0$, $n_1 (0) = 1$,
and $n_2 (0) = k \lambda$; however, we also have to take into account
that (because for $\ell =0$, $g = O(x)$) we cannot remove the constant
$B$.

Using this argument inductively, removing all possible terms of
homogeneity $k, k+1,\dots k +\ell$ and taking the formal limit as
$\ell \to \infty$ of the associated coordinate changes, we obtain the
following.

\begin{proposition}
  \label{pro:normal2} Let $X$ be generic, with
  $X_k = A z w^k \dop{z} + B w^{k+1} \dop{w} + \dots$, where
  $A \neq 0$ and $B\neq 0$, and let
$$\lambda = B/A,\,\,\mu:=1/\lambda.$$ 
Set $$n_1(\ell) = 1 - \ell \lambda, \,\, n_2(\ell) = (k - \ell) \lambda.$$
Then there exists a formal change of coordinates of the form
$(z,w) \mapsto ((z+ O(w), w +O(w^2) )$ such that in the new
coordinates, $X$ can be written (after rescaling) as
\begin{equation}\label{e:normal2} X = \left(\mu z +
    \sum_{\substack{\ell > 0 \\ n_1 (\ell) \in \N} } c_{\ell} z^{n_1
      (\ell)} w^{\ell}\right)w^{k} \dop{z}
  + \left( 1+ \sum_{\substack{\ell \geq 0  \\ n_2 (\ell) \in \N } } d_{\ell} z^{n_2 (\ell)} w^{\ell} \right)w^{k+1} \dop{w}. \end{equation}
\end{proposition}

\begin{remark}\label{exceptcase}
  It is possible, by inspecting the proof of \autoref{pro:normal2}, to
  extend its assertion to the exceptional case $B=0$, and obtain the
  (formal) normalizations: \beq iz\hat F(w)\dz+\hat G(w)\dw \eeq for
  some formal power series $\hat F(w),\hat G(w)$ with $\hat F(0)=1$
  and $\hat G(0)=\hat G'(0)=0$ (taking also into account
  \autoref{B=0}).
\end{remark}
\begin{remark}\label{rational}
  It is visible from \autoref{pro:normal2} that if
  $\lambda \not\in \mathbb{Q}_-$, then the prenormal form
  \eqref{e:normal2} becomes complete:
  \begin{equation}\label{complete}
    X =  \mu z w^k\dop{z} 
    +w^{k+1} \dop{w}, \qquad \mu := \frac{1}{\lambda}. 
  \end{equation}
  Note that this normal form coincides with the ones by
  Lombardi-Stolovitch \cite{LS} ($k>0$) or Poincar\'e-Dulac ($k=0$)
  \cite{ilyashenko}). If we assume that $\im\mu\neq 0$ then all the
  "small divisors"
 $$S=\Bigl\{|\mu-m_1\mu-m_2|,\,|k+1-m_1\mu-m_2|\Bigr\}_{m_1+m_2\geq 2,\, m_1,m_2\geq 0}$$
 are separated away from $0$, and so the convergence results in
 \cite{ilyashenko,LS} imply that one can choose a {\em convergent}
 transformation to the normal form \eqref{complete}.
\end{remark}
Note that there exists, in principle, {\em many} distinct
transformations bringing to the formal normal form
\eqref{e:normal2}. We demonstrate below that one can in fact always
choose a {\em biholomorphic} transformation to (some) normal form of
the kind \eqref{e:normal2}.

In view of \autoref{pro:normal2} it is natural to introduce yet
another weight system,
based on the fact that
$n_1 (\ell)$ and $n_2(\ell)$ in \eqref{e:normal2} satisfy
$n_1(\ell) = 1 - \ell \lambda$, $n_2(\ell) = (k - \ell) \lambda$;
thus,

if we assign the weight  $ \mu$ to $z$ and $1$ to $w$, each of the 
monomials $z^{n_1(\ell)} w^\ell$ has weight $\mu$ and $z^{n_2 (\ell)} w^\ell$ has weight $k$. Hence we arrive at the following decomposition (partial normal form) into vector fields homogeneous with 
respect to this new weight system: 
\begin{equation}
\label{k2k} X = Z_k + Z_{2k}, 
\end{equation} 
where we set 
\[ \begin{aligned}
Z_k &= \left(\mu z+\sum_{\substack{\ell \geq 0 \\ n_1 (\ell) \in \N} } c_{\ell} z^{n_1 (\ell)} w^{\ell}\right)w^{k}  \dop{z} = \mu \gamma(z,w) z  w^{k} \dop{z} + 
  w^{k +1} \dop{w} \\ \\
Z_{2k} &= \left( \sum_{\substack{\ell \geq 0  \\ n_2 (\ell) \in \N } } d_{\ell} z^{n_2 (\ell)} w^{\ell}\right) w^{k+1}\dop{w} = \delta(z,w) w^{k +1} \dop{w},
\end{aligned} \]
with $\gamma(t^{\mu} z, t w) = \gamma(z,w)$, and
 $\delta(t^{\mu} z, t w) = t^{k} \delta (z,w)$.

\begin{remark}
The partial normalization bringing $X$ to the normal form \eqref{k2k} {\em does} destroy in principle the special coordinates \eqref{mnonminimal}; this will be taken care of in the next section. 
\end{remark}

\section{The complete normalization in the generic case}

We continue with the assumptions and notations of the last section, and assume, in addition, that
\begin{equation}\label{rat}
\lambda = \mu^{-1} \in\mathbb Q^-
\end{equation}
(according to \autoref{rational}).

We start  by proving the following useful proposition about formal 
biholomorphisms keeping the partial normal form intact.  

\begin{proposition}\label{pro:deletion1} Assume that $X$ is in the prenormal form \eqref{e:normal2} or equivalently \eqref{k2k}. If 
 $\varphi(z,w)$ is homogeneous of 
degree $\mu$ with $\varphi(z,0) = 0$, and  $\nu(z,w)$ is homogeneous of degree $k$ (with $\nu(0,w) = 0$ if $k = 1$), then there
exists a unique formal biholomorphism of the form $h=(z+ \varphi + f, w + \nu + g)$ with $f$ containing
only terms of homogeneity exceeding $\mu$, and 
$g$ containing only terms of homogeneity exceeding 
$k$ such that $h_* X$ is again in normal form. 
\end{proposition}

 \begin{proof} We  study the transformation of the the vector field $X$ under a 
 transformation of the form $(z^*,w^*) \mapsto (z^* + f(z^*, w^*), w^* + g(z^*, w^*))$, with $f$ and $g$ homogenous, of degree $ \ell + \mu$ and $\ell + 1$, respectively. This is in our new grading, i.e. $f(t^\mu z^* , t w^*) = t^{\ell+\mu}f(z^*, w^*) $
 and $g(t^\mu z^* , t w^*) = t^{\ell +1}g(z^*, w^*) $, and we do not assume that $\ell$ 
 is an integer here. This means that  
 the vector field 
 \[ H = f(z^*, w^*) \dop{z^*} + g(z^*, w^*) \dop{w^*}  \]
 is homogeneous of degree $ \ell$. 

 Note that up to terms of homogeneity $> k + \ell$, we have that in the new
 coordinates, 
 \[ \begin{aligned}
 Z^* &= \lambda z^* (w^*)^{k} \dop{z^*} +  (w^*)^{k + 1} \dop{w^*} \\ 
 & \quad + \lambda (f(z^*, w^*) (w^*)^{k} + k z^* (w^*)^{k - 1} g(z^*, w^*) )  \dop{z^*} + (k+1)(w^*)^{k + 1}g(z^*, w^*) \dop{w^*} \\ &\quad -
 (\lambda z^* (w^*)^{k}  f_{z^{*}}   +   (w^*)^{k + 1} f_w^*) \dop{z^*} 
  - (\lambda z^* (w^*)^{k}  g_{z^{*}}   +  (w^*)^{k + 1} g_w^*) \dop{w^*} \\
  & = Z^* + [H, Z^*]. 
 \end{aligned}  \] 
 But since $H$ is homogeneous of degree $\ell$ with respect 
 to the weight system $(\mu, 1)$,  we have 
 \[ \begin{aligned} {} [H, Z^*]
  & = \lambda (f(z^*, w^*) (w^*)^{k} + k z^* (w^*)^{k - 1} g(z^*, w^*) )  \dop{z^*} + (k+1)(w^*)^{k + 1}g(z^*, w^*) \dop{w^*} \\ &\quad -
 (\lambda z^* (w^*)^{k}  f_{z^{*}}   +  (w^*)^{k + 1} f_w^*) \dop{z^*} 
  - (\lambda z^* (w^*)^{k}  g_{z^{*}}   +  (w^*)^{k + 1} g_w^*) \dop{w^*} \\
  &= \lambda (f(z^*, w^*) (w^*)^{k} + k z^* (w^*)^{k - 1} g(z^*, w^*) )  \dop{z^*} + (k+1)(w^*)^{k + 1}g(z^*, w^*) \dop{w^*} \\ &\quad -
 (\lambda z^*   f_{z^{*}}   +   w^* f_w^*)(w^*)^{k} \dop{z^*} 
  -  (\lambda z^*   g_{z^{*}}   +   w^* g_w^*) (w^*)^{k} \dop{w^*} \\
  &= \lambda (f(z^*, w^*) (w^*)^{k} + k z^* (w^*)^{k - 1} g(z^*, w^*) )  \dop{z^*} + (k+1)(w^*)^{k + 1}g(z^*, w^*) \dop{w^*} \\ &\quad -
 ((\ell + \mu) f)(w^*)^{k} \dop{z^*} 
  -  ((\ell+1 )g ) (w^*)^{k} \dop{w^*} \\
  & = \left( -  \ell w^* f + k z^*  g \right) (w^*)^{k - 1} \dop{z}
  -  ((k - \ell )g ) (w^*)^{k} \dop{w^*}.
 \end{aligned} \]

Recall that $\mu \in \mathbb{Q}$. We order 
the attained homogeneities by size, 
$k < k_1 < \dots$. Let us assume that we have
already constructed $h_p$ for some $p\in\N$ such 
that 
\[h_p = \begin{cases}(z + \varphi + f_p, w ) & k_p < 2 k \\
(z + \varphi + f_p, w + \nu + g_p) & k_p \geq 2 k
\end{cases} \]
with $(h_p)_* X $ in normal form up to homogeneous
terms of order exceeding $k_p$; denote
the terms of homogeneity 
exactly $k_{p+1}$ in $(h_p)_* X$ by $R \dop{z} + S \dop{w}$. We then construct 
$(f,g)$ by solving the equations 
\[
\begin{aligned}
(- B k_{p+1}  f(z,w) + k z  g(z,w)) w^{k} &= -  R \\ 
  - B (k - k_{p+1} )g (z,w)  w^{k}  &= - S 
\end{aligned}
\]
with $g$ homogeneous of degree $k_{p+1} - k$
and $f$ homogeneous of degree $k_{p+1} - k + \mu$.

By the computation above, we now see that $h_{p+1} = (z + \varphi + f_p +f , w + \nu + g_p + g) $ (with 
the obvious adaptation for small $p$) has the 
property that $(h_{p+1})_* X$ is in normal form 
up to terms of homogeneity exceeding $k_{p+1}$.
We observe that $h = \lim_{p \to \infty} h_p$ 
exists as a formal biholomorphism and conclude the statement of the proposition. 
\end{proof}

Assume now again that $X$, given in some holomorphic coordinates \eqref{k2k}, is an infinitesimal automorphism of an $m$-nonminimal hypersurface $M$ through $0$. We decompose the defining function $v = \psi(z,\bar z, u)$ of $M$  into terms which are homogeneous with respect to 
the above grading $[z]=\mu,\,[w]=1$. The corresponding index set of possible homogeneities is $\Gamma = \mu \mathbb{N} + \mathbb{N} $ and we write
\[ \psi(z,\bar z,u) = \sum_{\ell \in \Gamma} \psi_\ell (z,\bar z, u ), \quad 
\psi_\ell (t^\mu z, t^\mu \bar z, t u) = \ell \psi_\ell (z,\bar z, u ). \] The next Lemma is a key property. 

\begin{lemma}
\label{lem:startingterm} Assume that $X = Z_{k} + Z_{2 k}$ is an infinitesimal automorphism in the form \eqref{k2k} of the nonminimal real hypersurface through $0$
defined by $v = \psi(z,\bar z, u) = \sum_\ell \psi_\ell (z,\bar z,u)$. Then $\psi_\ell = 0$ for $\ell \leq k$. 
\end{lemma}
\begin{proof} We write as before
 \[ \begin{aligned}
Z_k &=  \mu \alpha(z,w) z  w^{k} \dop{z} + 
  w^{k +1} \dop{w} \\ \\
Z_{2k} &= \delta(z,w) w^{k +1} \dop{w},
\end{aligned} \]
with $\alpha $ homogeneous of degree $0$, 
$\delta$ homogeneous of degree $k$, and
\[ \psi = \sum_p \theta_p (z,\bar z,u),  \quad \theta_p (tz, t \bar z,u) = t^p \theta(z,\bar z, u),  \]
and prove by induction on $p$ that $\theta_p$ vanishes to order at least $k  - \mu p +1 $ in $u$. This immediately implies the assertion of the lemma.

 We need to make use of the 
tangency equation 
\begin{equation}\label{e:tangencyweighted}\begin{aligned}
0&= \real \left( X (v - \psi(z,\bar z, u) ) \right)|_{v = \psi(z, \bar z,u)} \\
&= \real \left( \left( 1 + \delta(z,u+i \psi)\right) (u+ i \psi)^{k+1} \left( \frac{1}{2i} - \frac{1}{2}\psi_u  \right) - \mu \alpha (z,u+ i \psi) z (u+ i \psi)^{k } \psi_z   \right) \\ 
&= \real \left( (1+\delta) \frac{u^{k+1}}{2i} - \mu \alpha z u^k \psi_z  +
(1+ \delta)  \frac{(k+1) u^{k} \psi}{2}\right)   
\end{aligned}
\end{equation}

The starting point is the case $p=0$. Evaluating \eqref{e:tangencyweighted} for $z = 0$, and writing $\theta_0 (u)$ instead of $\theta_0 (z,\bar z, u)$, gives 
\[ 0= \real \left( \left( 1 + (u+ i \theta_0 (u))^{k} d_{k}\right) (u+ i \theta_0 (u))^{k} \left( \frac{1}{2i} - \frac{1}{2} \theta'_0 (u)  \right)  \right). \]
Let $\theta_0 (u) = u \tilde \theta_0 (u)$. Then we have 
\[ \begin{aligned}
0&= \real \left( \left( 1 + u^{k} (1+ i \tilde \theta_0 (u))^{k} d_{k}\right) (1+ i \tilde \theta_0 (u))^{k+1} u^{k+1} \left( \frac{1}{2i} - \frac{1}{2}( u \tilde \theta'_0 (u) + \tilde \theta_0 (u)  )\right)  \right) \\
& = u^{k+1} \left(  -\frac{1}{2} u \tilde \theta'_0 (u) + \frac{k}{2} \tilde\theta_0 (u) + \dots  \right),
\end{aligned} \]
with the $\dots$ signifying nonlinear terms (in $u, u \tilde \theta'_0 (u), \tilde\theta_0(u)$). It follows that $\tilde \theta_0 $ vanishes to order $k -1$, 
and hence $\theta_0$ vanishes to order $k$ as claimed.

In order to see the induction step, consider now the terms on the right hand side of \eqref{e:tangencyweighted} which are of some 
degree $p$ in $z$, assuming that $\ord_u \theta_j (z,\bar z,u) \geq k - \mu j +1 $ for $j\leq p-1$. The contributions containing only terms of the form $\theta_j$ 
with a $j\leq p-1$ are necessarily nonlinear in those terms. The general such nonlinear term only contains monomials 
in $z,\bar z, u$ of joint homogeneity exceeding $k$ by induction (here we use that $k\geq 1$). Therefore, it is enough to 
consider the equation involving $\theta_p$ (necessarily linearly so), which becomes (after replacing $\theta_p$ by $\tilde \theta_p$ with $\theta_p (u) = u \tilde \theta_p (u)$ and cancelling $u^{k + 1}$)
\[ \frac{1}{2} \left(  k  \tilde \theta_p (u) -  u \tilde \theta_p' (u) - \mu z \tilde  \theta_{p;z} (u) - \bar \mu \bar z \tilde \theta_{p, \bar z} \right)   = \dots,\]
where the $\dots$ stands for terms vanishing to the correct order. It follows that the first nonzero term in the expansion of $\tilde \theta_p$ occurs 
with $u^{k - \mu p} $ as claimed. 
\end{proof}

Our next goal is 
\begin{proposition}
\label{l:deletemost} Let $M$ be a nonminimal 
hypersurface through $0$, $X \in \aut (M,0)$
be of generic type, and $\mu$ as above. Then there exist formal coordinates
such that $\psi$ satisfies
\[ \psi (z,0,u)  = \bar \psi(0,\bar z, u) = 0 
\text{ and } \tilde \psi_{k} (z, \bar z , u) =
C z^p (\bar z)^q u^{k - \mu (p+q)} + O (z^{p+1}).  \]
In these coordinates,  $X$ (after 
possibly rescaling by a real number) takes the form
\begin{equation}\label{nfgen}
X= \mu zw^{k} \dop{z} + w^{k+1} \dop{w} + rw^{2k+1} \dop{w}, \quad r\in\R,   
\end{equation}
\end{proposition}
\begin{proof}
We compute the term of homogeneity $2 k +1$ in \eqref{e:tangencyweighted}:
\[
0 = u^{k +1} \real \left(   \frac{\delta(z,u)}{2 i} +  \frac{k }{2} \tilde \psi_{k} - \frac{1}{2} u \tilde\psi_{k;u}  - \mu  \alpha (z,u) z \tilde\psi_{k;z} \right).
\]
In other words, we have that 
\begin{equation}\label{e:reflect} -i \delta(z,u) +i \bar \delta (\bar z, u) + k \tilde\psi_{k} - u \tilde\psi_{k;u} - \mu  \alpha (z, u) z \tilde \psi_{k;z} - \mu  \bar \alpha (\bar z, u) \bar z \tilde \psi_{k;\bar z} = 0, \end{equation}
as a power series in $z, \bar z, u$. We set $\bar z = 0$ to obtain
\begin{equation}
 \label{e:reflect1} -i \delta(z,u) +i \bar \delta (0, u) + k \tilde\psi_{k} (z,0,u) - u \tilde\psi_{k;u} (z,0,u) - \mu  \alpha (z, u) z \tilde \psi_{k;z} (z,0,u) = 0.
 \end{equation} 
 This equation suggests to make the following change of coordinates: $$\Phi(z^*,w^*) = (z^*, w^*(1 + 2 i  \tilde \psi_{k} (z^*,0,w^*)).$$ Since the 
 coordinate is of 
 the form considered in \autoref{pro:deletion1}, we can apply that Proposition to 
 see that with an appropriate change
 of coordinates $$\Phi (z^* , w^*) = (z^* + f, w^* + 2i \tilde \psi_{k} + g),$$  $X$ stays in normal form, while in the new coordinates $(z^*, w^*)$, $M$ is 
 defined by  
 \[  \imag w^* +  \psi_{k} (z^*, 0, w^*)  +  \psi_{k} (0, \bar z^*, \bar w^*) =  \psi(z^*, \bar z^*, u^*  ) + \dots .  \]
 Hence, after this change of coordinates, if we drop the $*$ from the notation again, we can assume that $M$ is 
 defined by an equation of the form $v = u \tilde \psi (z,\bar z,u)$ with $\tilde \psi_{k} (z,0,u) = \tilde \psi_{k} (0,\bar z,u) = 0$. In these 
 coordinates, we therefore have by \eqref{e:reflect1} that $\delta(z,u) = r u^{k}$ for some real number $r$. 

 So let us return to \eqref{e:reflect}. We decompose $ \tilde \psi_{k} $, which in our new coordinates does not contain any pure terms, as 
 \[ \tilde \psi_{k} (z,\bar z, u) = C z^p ( \bar z^q u^{k - \mu (p+q)} + O (\bar z^{q+1})) + O(z^{p+1}). \]

 We claim that there is a change of coordinates of the form $$ \Psi(z^* , w^*) = (z^* + F(z^*, w^*), w^*),$$
 with $F$ being homogeneous of degree $\mu $ with respect to the $(\mu ,1)$-grading, such that in the new coordinates,  $$\tilde \psi^*_{k} ({z^*},{\bar z}^*, u) = C {z^*}^p  ({{\bar z}^*})^q u^{k - \mu  (p+q)}  + O({z^*}^{p+1}) .$$ 
 Let $\mu  = \frac{m}{n} $ with $\gcd({m,n}) = 1 $, and $m<0$, $n>0$. Then we can write
 \[ \tilde \psi_{k} (z,\bar z, u) = C z^p \bar z^q u^{k -  \mu   (p+q)} (\bar{\tilde{ \varphi}} (\bar z^{-m} u^n))^q + O (z^{p+1}), \]
 where $\tilde \varphi (s)$ is a power series in $s$ with $\varphi(0) = 1$. The coordinate 
 change $(z,w) \mapsto (z \tilde \varphi(z,w), w)$
 is again of the 
 form considered 
 in \autoref{pro:deletion1}. We therefore 
 set $$\Phi (z^*, w^*) = (z^* + \varphi(z^* , w^*) + f(z^*, w^*), w^*) $$ with $f$ from that Proposition (one sees that in case $\psi = 0$, we have that $g = 0$ in \autoref{pro:deletion1}). 

 After this change of coordinates, \autoref{e:tangencyweighted} becomes
 \[  k \tilde \psi_{k} -  u \tilde\psi_{k;u}  - \mu  \alpha (z,u) z \tilde\psi_{k;z} -\mu  \bar \alpha (\bar z,u) \bar  z \tilde\psi_{k; \bar z} = 0. \]
 Examining the coefficient of $z^p$ in this equation (as a power series in $\bar z$ and $u$)
 yields $\alpha(\bar z, u) = 1$. 
\end{proof}

The other side of the coin is the fact that the
tangency equation, for which $X$  is in 
the formal normal forms \eqref{complete} or \eqref{nfgen} above, and every 
starting data $\tilde \psi_{k}$, yields a unique
formal hypersurface to which $X$ is tangent. 

\begin{proposition}
\label{pro:uniquehypersurface} Assume that 
$X$ is as in \eqref{complete} or \eqref{nfgen} (with $\mu\in\CC{*}$),
and that $\varphi (z, \bar z, u)$ satisfies 
\[\varphi \left(  t^{ \mu  } z, t^{\bar  \mu  } \bar z, t u \right) = t^{k} \varphi (z, \bar z, u),\]
\[ \varphi (z,0,u) = \varphi (0,\bar z, u) = 0,  \]
as well as 
\[ \varphi (z, \bar z, u) = C z^p (\bar z^q u^{k -  \mu  (p+q)} + O(\bar z^{q+1}) ) + O (z^{p+1}). \]
Then there exists a unique real-valued formal power series $\tilde \psi (z, \bar z, u)$ such that for $M$
defined by $v = u \tilde \psi (z,\bar z,u)$ we have $X\in \hol (M,0)$ and 
such that the decomposition 
\[ \tilde \psi (z,\bar z, u) = \sum_{j} \tilde \psi_j (z, \bar z ,u ), \quad 
\psi_j \left(  t^ \mu   z, t^{\bar  \mu  } \bar z, t u \right) = t^{j} \psi (z, \bar z, u)  \]
satisfies $\tilde\psi_{k} = \varphi$, and $\tilde\psi_{j} \neq 0 $ only if $j = \ell k$ for some positive $\ell \in \N$. 
Furthermore, $\tilde \psi $ satisfies $\tilde \psi (z,0,u) = \tilde \psi (0,\bar z,u ) = 0$, and $\tilde \psi(z,\bar z, u) = C z^p (\bar z^q u^{k -  \mu  (p+q)} + O(\bar z^{q+1}) ) + O (z^{p+1})$.
\end{proposition}
\begin{proof}
Write out the tangency equation for $X$ in 
the form of \autoref{l:deletemost}:
\[ \begin{aligned}
0 &= \real \left( \left((u+ i u\tilde \psi)^{k+1} + r (u +i u \tilde\psi)^{2k+1}\right) \left(  \frac{1}{2i}
- \frac{1}{2} (u \tilde \psi_u + \tilde \psi)  
 \right) - A z (u + i u \tilde \psi)^{k} u \tilde \psi_z    \right) \\
 &= u^{k +1} \real \left( \left((1+ i \tilde \psi)^{k+1} + r u^{k} (1 +i  \tilde\psi)^{2k+1}\right) \left(  \frac{1}{2i}
- \frac{1}{2} (u \tilde \psi_u + \tilde \psi)  
 \right) - A z (1 + i  \tilde \psi)^{k}  \tilde \psi_z    \right) \\
\end{aligned} \]
This means that $\real X$ is tangent to the hypersurface defined by $\imag w = \real w \tilde \psi (z,\bar z, \real w)$ 
if and only if 
\[ \real \left( \left((1+ i \tilde \psi)^{k+1} + r u^{k} (1 +i  \tilde\psi)^{2k+1}\right) \left(  \frac{1}{2i}
- \frac{1}{2} (u \tilde \psi_u + \tilde \psi)  
 \right) - A z (1 + i  \tilde \psi)^{k}  \tilde \psi_z    \right) = 0. \]
We can read this equation as an equation inductively
determining $\tilde \psi_{k + \ell}$ for $\ell>0$
from $\tilde \psi_{k}$: If we examine 
the term of homogeneity $ k + \ell  $ in it, we 
can rewrite this term as 
\[  \begin{aligned}
P_\ell (u^{k}, \tilde{\psi}_{k+j},z \tilde{\psi}_{k+j; z}, \bar z \tilde{\psi}_{k+j ; \bar z} , u \tilde{\psi}_{k+j}\colon j<\ell ) 
+ k \tilde{\psi}_{k+\ell} -   (  u \tilde{\psi}_{k +\ell;u} + A z \tilde{\psi}_{k+\ell;z}  + \bar  A  \bar z \tilde{\psi}_{k+\ell;\bar z}) \\ 
= P_\ell (u^{k}, \tilde{\psi}_{k+j},z \tilde{\psi}_{k+j; z}, \bar z \tilde{\psi}_{k+j ; \bar z} , u \tilde{\psi}_{k+j}\colon j<\ell ) 
 -\ell  \tilde{\psi}_{k+\ell},
\end{aligned}   \]
where $P_\ell $ is a polynomial in its arguments such that
 $P_\ell (u^{k}, \tilde{\psi}_{k+j},z \tilde{\psi}_{k+j; z}, \bar z \tilde{\psi}_{k+j ; \bar z} , u \tilde{\psi}_{k+j}\colon j<\ell )$ 
 is homogeneous of degree $k + \ell$. It follows by induction that we can solve for $\tilde \psi_{k + j}$ and that 
 furthermore,  $P_{j} = 0  $ unless $k = \ell j$ for some $\ell \geq 1$. 

 The last two assertions of the proposition follow easily from this inductive procedure. 
\end{proof}

Now \autoref{rational}, \autoref{l:deletemost}, and \autoref{pro:uniquehypersurface} read together imply 
\begin{proposition}\label{nfgeneric}
Every vector field $X$ tangent to an infinite type hypersurface of generic type can be brought by a formal invertible  transformation to either the normal form \eqref{complete} for $\mu\not\in\mathbb Q^-$ or the normal form \eqref{nfgen} for $\mu\in\mathbb Q^-$.  All the parameters present in the normal form \eqref{nfgen} actually occur, i.e. for each collection of $k\geq 0,\mu\in\CC{*},r\in\RR{}$ there exists a real-analytic Levi-nonflat hypersurface with $X\in\mathfrak{aut}\,(M,0)$. 
\end{proposition}

It remains to deal with the convergence of the normal form \eqref{nfgen} for $\mu\in\mathbb Q^-$ (for $k=0$, one should employ \eqref{complete} instead of \eqref{nfgen}). To avoid case distinction, we assume for the rest of the section that $k>0$, as for $k=0$ the argument is in fact simpler.  Let us write
$$\mu=-\frac{p}{q},\,\,p,q\in\mathbb N,$$
so that after a real scaling we have
$$X=w^k\left[-(pz+...)\dop{z}+(qw+...)\dop{w}\right]$$ for the original vector field, and 
\beq\label{nfpq}
X_N=-pzw^k\dop{z}+(qw^{k+1}+rw^{2k+1})\dop{w}
\eeq
for the normal form. We now prove the following statement, which is crucial for the convergence property.
\begin{proposition}\label{symmetries}
Assume that a generic type vector field $X$ with $\mu\in\mathbb Q^-$ can be formally brought to a normal form $X_N$ of the kind \eqref{nfpq}. Then any other normal form \eqref{e:normal2} of $X$ also equals to $X_N$. 
\end{proposition}
\begin{proof}
Let us start with computing sufficiently many {\em formal symmetries} of $X_N$, i.e. formal transformations of $(\CC{2},0)$ mapping $X_N$ into itself. We note that any (formal) vector field $L,\,L(0)=0,$ commuting with $X_N$ generates a flow of (formal) symmetries of $X_N$ (e.g., \cite{ilyashenko}). We now take any formal vector field 
$$L=f(z,w)\dop{z}+wg(z,w)\dop{w}$$
vanishing at the origin, and work out the commutation condition $[X_N,L]=0$. It is straightforward to compute that this results in the following system of PDEs for $f,g$:
\beq \label{commute}
\begin{aligned}
&-pzf_z+(qw+rw^{k+1})f_w+pf+kpzg=0,\\
&-pzg_z+(qw+rw^{k+1})g_w-kg(q+2rw^k)=0.
\end{aligned}
\eeq
As a well known fact (e.g. \cite{ilyashenko}), there exists a local biholomorphic transformation 
\beq\label{dim1}
w\mapsto w+O(w^{k+1})
\eeq
of $(\CC{},0)$, mapping the $1$-dimensional polynomial vector field $(qw+rw^{k+1})\dop{w}$ into its normal form $qw\dop{w}$. After applying it, the system \eqref{commute} transforms into
\beq \label{commute1}
\begin{aligned}
&-pzf_z+qwf_w+pf+kzg=0,\\
&-pzg_z+qwg_w-qkg\left(1+...)\right)=0,
\end{aligned}
\eeq 
where dots stand for terms $w^j,\,j\geq k$.
We then perform the blow-up transformation
\beq\label{blowup}
z^*=z^qw^p,\quad w^*=w,
\eeq
which maps the linear vector field 
$$L_0:=-pz\dop{z}+qw\dop{w}$$
into the monomial linear vector field $qw\dop{w}$. Then the second PDE in \eqref{commute1} becomes independent of $z$, precisely it becomes
$$wg_w=k(1+...)g$$
(dots again stand for an analytic expression in $w$ vanishing at $0$). Solving the latter PDE (as a linear ODE for $g(z,w)$ for each fixed $z$) and performing the inverse substitution w.r.t. \eqref{blowup}, we compute 
\beq\label{findg}
g(z,w)=C(z^qw^p)w^{k}(1+...),
\eeq
where dots stand for a (fixed) analytic power series in $w$ vanishing at $0$, and $C(z^qw^p)$ is an arbitrary power series in one variable. Substituting the outcome for $g$ into the first equation in \eqref{commute1}, we get:
\beq\label{findf}
L_0(f)+pf=-kzC(z^qw^p)w^{k}(1+...).
\eeq 
Note that each monomials $z^aw^b$ is an eigenvector of the operator 
$$f\mapsto L_0(f)+fp$$
(considered on the linear space of all formal power series in $z,w$) with the eignvalue $qb-p(a-1)$. It follows from here that all the monimials $z^aw^b$ with $qb-p(a-1)\neq 0$ (we call them {\em non-resonant}, while the other {\em resonant}) lie in the image of the operator. It is easy to see (from $k>0$ and the fact that dots stand for terms in $w$ only) that the right hand side in \eqref{findf} does {\em not} contain the latter monomials, thus for each $g$, as in \eqref{findg}, there exists a unique solution of \eqref{findf}, which has the same representation as  the right hand side of \eqref{findf} has (with a different series $C(\cdot)$). A solution of the linear PDE \eqref{findf} is defined uniquely up to a series consisting of resonant monomials only. Such a series clearly has the form $zD(z^qw^p)$ for some formal power series $D$ in one variable. We finally conclude that 
\beq\label{findf1}
f(z,w)=z\Bigl(D(z^qw^p)+\tilde C(z^qw^p)w^k(1+...)\Bigr)
\eeq 
(here $\tilde C,D$ are power series in one variable; $\tilde C$ is uniquely determined by $C$, while $D$ is arbitrary).
It remains to deal with the inverse substitution for \eqref{dim1}. However, it is easy to see that the representations \eqref{findg},\eqref{findf1} remain valid, just with a different $C,\tilde C,D$. 

We now compare the outcome of our calculation to the normalization procedure in the proof of \autoref{pro:normal2}. Over there, in our notations, a normalizing transformation 
$$z\mapsto z+F(z,w),\quad w\mapsto w+G(z,w)$$ is determined uniquely modulo resonant monomials in $F,G$, which have respectively the form 
\beq\label{resterms}
zA(z^qw^p) \,\,\,\mbox{or}\,\,\,w^kB(z^qw^p)
\eeq for arbitrary power series $A,B$ in one variable vanishing at $0$. However, flows of formal vector fields with coefficients $f,g$ satisfying \eqref{findg},\eqref{findf1} (all of which must preserve the normal form $X_N$ in view of the above) have the same number of free parameters in each finite jet as normalizing transformations of $X_N$ do. This can be seen, for example, from the exponential representation
\beq\label{expon}
H^t=I+tX+\sum_{j=2}^\infty \frac{1}{j!}t^jX^j
\eeq
of the formal flow (for employing \eqref{expon}, one shall understand a formal vector field as a derivation of the algebra of formal power series, and $X^j$ as a composition of derivations; see, e.g., \cite{ilyashenko}).
Thus, the latter flows generate {\em all} the normalizing transformations of $X_N$, and this already  implies the assertion of the proposition for $k>0$.

For $k=0$, the proposition is proved identically, with just employing the normal form \eqref{complete} instead of \eqref{nfgen} and without the need to perform the substitution \eqref{dim1}.   
 
\end{proof}

We are in the position now to prove the convergence theorem. 
 
\begin{theorem}\label{converge}
Let $X$ be a holomorphic vector field near the origin of a generic type and $\mu=\beta_k/\alpha_k'(0)=B/A$ as above. Assume that $\re X$ is tangent to a real-analytic Levi-nonflat infinite type hypersurface $M$ through $0$. Then there exists a biholomorphic transformation bringing it to the normal form \eqref{complete} for $\mu\not\in\mathbb Q^-$ or the normal form \eqref{nfgen} for $\mu\in\mathbb Q^-$.  
\end{theorem}
\begin{proof} As discussed in \autoref{rational}, for $\mu\not\in\mathbb Q^-$ the assertion  of the theorem follows from the Poincar\'e-Dulac or Lombardi-Stolovitch theory (actually, an independent proof can be obtained by repeating the proof below for the situation $\mu\in\mathbb Q^-$, but we don't need these extra considerations and leave up to the interested readers). Hence, we deal with the case $\mu\in\mathbb Q^-$ and use the normal form \eqref{nfgen} (or \eqref{complete} for $k=0$, but we again do not consider this case to avoid case distinction, since the proof is then completely analogous). In view of \autoref{nfgeneric}, there exists a formal transformation
$$z\mapsto z+F(z,w),\quad w\mapsto w+wG(z,w)$$
of $X$ to a normal form $X_N$, as in \eqref{nfgen}. Furtheremore, in view of \autoref{symmetries}, this transformation can be chosen in such a way that all the resonant monomials \eqref{resterms} in $F,G$ {\em vanish for its inverse}. We fix the latter transformation and denote {\em its inverse} by $$H_0=(z+F,w+wG).$$ 

We now write down the fact that $H_0$ maps $X_N$ into $X$. Then a direct computation gives the following system of PDEs:
$$pz(1+F_z)-(qw+rw^{k+1})F_w=p(z+F)(1+G)+...,$$
$$pzG_z-(q+rw^{k})(1+G+wG_w)=-\bigl((1+g)^{k+1}+rw^k(1+g)^{2k+1}\bigl)+...,$$
where dots stand for an analytic expression in $z,w,F,G$, vanishing at the origin together with its differential. This immediately gives:
\beq\label{thePDEs}
\begin{aligned}
&-pzF_z+(qw+rw^{k+1})F_w+pF=...,\\
&-pzG_z+(qw+rw^{k+1})G_w-kqG=....
\end{aligned}
\eeq
Here dots, again, stand for an analytic expression in $z,w,F,G$, vanishing at the origin together with its differential. We shall now proceed similarly to the proof of \autoref{symmetries} and perform the local biholomorphic transformation \eqref{dim1}, mapping the $1$-dimensional polynomial vector field $(qw+rw^{k+1})\dop{w}$ into its normal form $qw\dop{w}$. Then, in the new coordinates, \eqref{thePDEs} reads as
\beq\label{PDEs1}
\begin{aligned}
&-pzF_z+qwF_w+pF=A(z,w,F,G),\\
&-pzG_z+qwG_w-kqG=B(z,w,F,G).
\end{aligned}
\eeq
Here $A,B$ are analytic expression in $z,w,F,G$, vanishing at the origin together with their differential.

We shall now prove the convergence of the transformation $H_0$, based on the system of PDEs \eqref{PDEs1}. 
We expand
$$F=\sum_{\alpha+\beta\geq 2}F_{\alpha\beta}z^\alpha w^\beta,\quad G=\sum_{\alpha+\beta\geq 2}G_{\alpha\beta}z^\alpha w^\beta$$ and substitute it into \eqref{PDEs1}. Then, fixing the order $m:=\alpha+\beta$, we obtain equations of the kind:
\beq\label{homol}
(-p\alpha+q\beta+p)F_{\alpha\beta}=P_{\alpha\beta},\quad
(-p\alpha+q\beta-kq) G_{\alpha\beta}=Q_{\alpha\beta}
\eeq
where $P_{\alpha\beta},Q_{\alpha\beta}$ are {\em universal} polynomials in (i) the Taylor coefficients $F_{\alpha\beta},G_{\alpha\beta}$ with $\alpha+\beta<m$; (ii) Taylor coefficients of $A,B$ at the origin  of order $\leq m$. (Note that the coefficients of the polynomials $P_{\alpha\beta},Q_{\alpha\beta}$ are {\em nonnegative}). In particular, on the formal level, we see that we can uniquely solve \eqref{homol} for the Taylor coefficients of $F,G$, besides the coefficients with $$\mbox{either}\,\,\,-p\alpha+q\beta+p=0,\,\,\, \mbox{or}\,\,\, -p\alpha+q\beta-kq=0$$ (precisely corresponding to the resonant terms in $F,G$).  Recall that, by the definition of $H_0$, the latter resonant terms are already determined: 
\beq\label{nores}
F_{\alpha,\beta} =0,\,\,-p\alpha+q\beta+p=0;\quad   G_{\alpha,\beta}=0,\,\, -p\alpha+q\beta-kq=0.
\eeq 
Let us consider now a subseries in the series of equations \eqref{homol} obtained by "ingoring" all the resonant equations. Let us denote it by $(*)$.  We then use the {\em Cauchy majorants method.} Namely, we replace the system of equations $(*)$  by another system $(**)$ obtained from $(*)$ as follows:

\smallskip

 (i) we replace the coefficients $-p\alpha+q\beta+p,\,-p\alpha+q\beta-kq$  in the left hand side of $(*)$ (which are nonzero integers and thus have absolute value $\geq 1$) by the constant coefficient $1$;

  \smallskip
  
  (iii)  we replace the Taylor coefficients of $A,B$ by their absolute values, denoting the new series respectively by $|A|,|B|$. 
  
  \smallskip

We then solve the new system of equations uniquely, and obtain solutions 
$$\bigl\{\tilde F_{\alpha\beta}\bigr\},\,\, -p\alpha+q\beta+p\neq 0;\
\quad \bigl\{\tilde G_{\alpha\beta}\bigr\},\,\, -p\alpha+q\beta-kq\neq 0$$
(the remaining "resonant" Taylor coefficients we put simply zero).
 It is now straightforward to prove, by induction in $m=\alpha+\beta$ that: 

\smallskip

(I) all the solutions $\tilde F_{\alpha\beta},\tilde G_{\alpha\beta}$ of $(**)$ are nonnegative (this follows by induction from the fact that the right hand sides of $(**)$ consist of polynomials in the coefficients $\tilde F_{\alpha\beta},\tilde G_{\alpha\beta}$ with $\alpha+\beta<m$, which are already proved to be nonnegative); 

\smallskip 

(II) we have 
\beq\label{ineq}
|F_{\alpha\beta}|\leq \tilde F_{\alpha\beta}, \quad |G_{\alpha\beta}|\leq \tilde G_{\alpha\beta}
\eeq
(this follows by induction from the construction of the system $(**)$).  

\smallskip

Finally, we consider the series of coefficients $F^*_{\alpha\beta},\tilde G^*_{\alpha\beta}$, obtained by considering the functional equation
\beq\label{implicit}
F^*=|A|(z,w,F^*,G^*),\quad G^*=|B|(z,w,F^*,G^*),
\eeq
where $|A|,|B|$ are as above, collecting terms with $z^\alpha w^\beta$, and solving for the Taylor coefficients of $F^*,G^*$. We again easily prove by induction that, first, all the Taylor coefficients of $F^*,G^*$ are nonnegative, and second, we have:
\beq\label{ineq1}
\tilde F_{\alpha\beta}\leq F^*_{\alpha\beta}, \quad \tilde G_{\alpha\beta}\leq G^*_{\alpha\beta}.
\eeq 
Now $(F^*,G^*)$ is the unique holomorphic solution of \eqref{implicit} near the origin (provided by the implicit function theorem), and the inequalities \eqref{ineq}, \eqref{ineq1} imply finally the convergence of the series
$$\sum_{\alpha+\beta\geq 2}F_{\alpha\beta}z^\alpha w^\beta,\,\, \sum_{\alpha+\beta\geq 2}G_{\alpha\beta}z^\alpha w^\beta,$$
as required.

\end{proof}

\begin{remark}
We remark that, even for linear vector fields $pz\dop{z}-qw\dop{w},\,p,q\in\mathbb N$ (which are generic with $k=0$), there exist formal divergent automorphisms. That is why for $\lambda\in\mathbb Q^-$ {\em not} every normalizing transformation is convergent (but convergent ones still exist, according to \autoref{converge}).  
\end{remark}

\section{The two exceptional cases}

We deal finally with the two exceptional cases outlined in Section 2.

\subsection{Case $\mbox{ord}\alpha_k=\infty$} The first one, we recall, is the case $$\alpha_k(z)\equiv 0$$ in \eqref{key}. We hence have $\beta_k=B\neq 0$. We, first of all, prove the following 
\begin{proposition}\label{ordinf}
Assume that in the expansion \eqref{key} one has $\alpha_k\equiv 0$. Then $X$ is formally equivalent at the origin to a  vector field of the kind 
\begin{equation}\label{c(z)}
Z_c:=(B w^{k+1}+c(z)w^{2k+1})\dop{w}
\end{equation} for some formal power series $c(z)$. 
\end{proposition}
\begin{proof}
Let us subject the vector field $X$ to a (biholomorphic) transformation
$$z\mapsto f(z,w), \quad w\mapsto w+wg(z,w),\quad g(0,0)=0,\quad f_z(0,0)=1,$$
aiming to map it into a vector field $Z_c$. Since 
$$X=O(w^{k+1})\dop{z}+(Bw^{k+1}+O(w^{k+2}))\dop{w},$$
it is not difficult to compute that the fact that $X$ is mapped into $Z_c$ reads as:
\begin{equation}\label{Fuchs}
0=O(1)f_z+(B+O(w))f_w, \quad B(1+g)^{k+1}+c(f)w^k(1+g)^{2k+1}=O(w)g_z+(B+O(w))(1+wg_w+g).
\end{equation}
The system \eqref{Fuchs} has actually the form
$$f_w=U(z,w,f,f_z),\quad wg_w=V(z,w,f,g,g_z).$$
Furtheremore, importantly, the linear part of $V$ in $g$ has the form $kg$.
We can now {\em formally} solve \eqref{Fuchs} for the Taylor coefficients of $f,g$ in $z,w$ as follows: for each fixed $j$, we compare in \eqref{Fuchs} all coefficients of the kind $z^iw^j$ and find one-by-one all $f_{z^iw^j},g_{z^iw^j}$ (assuming all the derivatives if $z^iw^l$ with $l<j$ are already known). We can do so until $j<k$ (and the unspecified $c(z)$ is not involved in this scheme). This, however, fails to be possible for $j=k$, since in this case the $z^iw^k$ coefficient of $g$ arises linearly in the right hand side. We then choose a corresponding Taylor coefficient of $g$ arbitrarily, and then uniquely choose $c_{z^i}$ in order that the $z^iw^k$ terms in the second equation in \eqref{Fuchs} agree. This determines $c(z)$. After that, we go on with the scheme for $j>k$ and solve uniquely for all the Taylor coefficients in $z^iw^j$, as required. 
\end{proof}

As a consequence of \autoref{ordinf}, we conclude that {\em the decomposition \eqref{k2k} also holds for $X$} (after a rescaling) with $\alpha(z,w)\equiv 0$ and $\delta(z,w)=c(z)w^k$ (the respective $\mu=0$). We then argue identically to the proof of \autoref{lem:startingterm} and obtain
\begin{lemma}
\label{psik} Assume that $X$ as above is an infinitesimal automorphism of a nonminimal real hypersurface through $0$
defined by $v = \psi(z,\bar z, u)$. Then $\psi$ vanishes to order $k$ in $u$. 
\end{lemma}
We then obtain
\begin{proposition}\label{main1}
In the assumptions of \autoref{psik}, there exist formal coordinates
such that $\psi$ satisfies
\[ \psi (z,0,u)  = \bar \psi(0,\bar z, u) = 0 
\text{ and } \tilde \psi_{k} (z, \bar z , u) =
C z^p (\bar z)^q u^{k - \mu (p+q)} + O (z^{p+1}).  \]
In these coordinates,  $X$ (after 
possibly rescaling by a real number) takes the form
\begin{equation}\label{e:nfspec1}
X=  w^{k+1} \dop{w} + rw^{2k+1} \dop{w}, \quad r\in\R.
\end{equation}
\end{proposition}
\begin{proof}
The proof is very similar to that for \autoref{l:deletemost} with the only difference that the decomposition of the defining function of $M$ is now $$\psi(z,\bar z,u)=\sum_{j=k}^\infty \psi_j(z,\bar z,u),\,\,\psi_j(z,\bar z,u)=u^j\theta_j(z,\bar z)$$
(according to \autoref{psik}), and the appropriate weights are $[z]=0,\,[w]=1$ (actually with accordance to \eqref{k2k}). 
\end{proof}

We thus arrive to
\begin{proposition}\label{nfspec1}
Every vector field $X$ with $\alpha_k\equiv 0$ in \eqref{key} which is tangent to a real-analytic Levi-nonflat hypersurface  can be brought by a formal invertible  transformation to the normal form \eqref{e:nfspec1}.  All the parameters present in the normal form \eqref{e:nfspec1} actually occur, i.e. for each collection of $k\geq 0,r\in\RR{}$ there exists a real-analytic Levi-nonflat hypersurface with $X\in\mathfrak{aut}\,(M,0)$.
\end{proposition}
\begin{proof}
  It remains to prove the realizability of each vector field \eqref{e:nfspec1} as an infinitesimal automorphism. Realizing that in spirit, the variables $z$ should be treated
  as having weight $0$, we make the Ansatz
  \[ v= u^\ell A(z,\bar z, u) = u^\ell f(u),  \]
  thinking of $f$ as a power series in $u$ having coefficients which are functions of $(z,\bar z )$. The tangency equation becomes
  \[
    \begin{aligned}
      2\real X (v-u^\ell f(u)) & = 2\real \left( (w^{k+1}  + r w^{2k+1})\dop{w} \right)(v- u^\ell f(u)) \\
                              &= 2\real   (w^{k+1}  + r w^{2k+1}) ) \left( \frac{1}{2i} - \frac{1}{2} (\ell u^{\ell-1} f - u^\ell f_u ) \right)  \\
                               &= \imag  (w^{k+1}  + r w^{2k+1}) -  \real(w^{k+1} +r w^{2k+1})  (\ell u^{\ell-1} f - u^\ell f_u ) \\
      &=0,
    \end{aligned}
  \]
  if $w = u+ i u^\ell f(u)$. We compute that
  \[\begin{aligned} \imag w^{k+1} + r w^{2k+1} &= \sum_{j=0}^{\lfloor \frac{k+1}{2} \rfloor} (-1)^j \binom{k+1}{2j+1} u^{k+1+(\ell-1)(2j+1)} f(u)^{2j+1} \\
                                               & \quad +r  \sum_{m=0}^{\lfloor \frac{2k+1}{2} \rfloor} (-1)^m \binom{2k+1}{2m+1} u^{2k+1+(\ell-1)(2m+1)} f(u)^{2m+1} \\
    & = u^{k+\ell} A(u,f(u)) \end{aligned}
      \]
      and
   \[\begin{aligned} \real w^{k+1} + r w^{2k+1} &= \sum_{j=0}^{\lfloor \frac{k+1}{2} \rfloor} (-1)^j \binom{k+1}{2j} u^{k+1+(\ell-1)(2j)} f(u)^{2j} \\
                                               & \quad + r \sum_{m=0}^{\lfloor \frac{2k+1}{2} \rfloor} (-1)^m \binom{2k+1}{2m} u^{2k+1+(\ell-1)(2m)} f(u)^{2m} \\
    & = u^{k+1} B(u,f(u)) \end{aligned},
\]
where $A(u,f) = (k+1)f + O(u,f^2) $ and $B(u,0)=1$. We therefore get the following differential equation for $f$:
\[   (\ell f - u f_u) = \frac{A(u,f)}{B(u,f)} \] 
Treating the latter as an ODE for $f=f(u)$ for each fixed $z$, for any choice of a function $f(0) = c(z,\bar z)$, we claim that this Fuchsian (Briott-Bouquet) differential equation has a solution provided that $\ell = k+1$. First note that $\ell = k+1$ is forced if $f(0) \neq 0$.
Now note that with
the choice $\ell=k$, the formulas for $A$ and $B$ simplify to
\[
  \begin{aligned}
    A(u,f) &=  \sum_{j=0}^{\lfloor \frac{k+1}{2} \rfloor} (-1)^j \binom{k+1}{2j+1} u^{2jk} f^{2j+1} 
             + r \sum_{m=0}^{\lfloor \frac{2k+1}{2} \rfloor} (-1)^m \binom{2k+1}{2m+1} u^{(2m+1)k} f^{2m+1} = (k+1) f +O(u^k)   \\
    B(u,f) & =\sum_{j=0}^{\lfloor \frac{k+1}{2} \rfloor} (-1)^j \binom{k+1}{2j} u^{2jk} f^{2j} + r \sum_{m=0}^{\lfloor \frac{2k+1}{2} \rfloor} (-1)^m \binom{2k+1}{2m} u^{(2m+1)k} f^{2m} = 1 + O(u^k)   
  \end{aligned}
\]
so that the equation for $f$ can be rewritten in the form
\[ (k+1) f- u f_u =  R(u^k, f) = (k+1)f + O(u^k).  \]
It follows that this Fuchsian ODE is actually a true ODE in disguise, from which existence and convergence of $f$ with any given convergent initial data $f(0) = c(z,\bar z)$ follows. 
\end{proof}

\begin{theorem}\label{converge-spec}
Let $X$ be a holomorphic vector field near the origin satisfying $\alpha_k\equiv 0$ in \eqref{key} which is tangent to a real-analytic Levi-nonflat infinite type hypersurface $M$ through $0$. Then there exists a biholomorphic transformation bringing it to the normal form \eqref{e:nfspec1}.
\end{theorem}
The proof of the theorem is very analogous to that for \autoref{converge} in the generic case, and we leave the details for the reader.

\medskip

\subsection{Case $B=0$} We finally consider the remaining exceptional case $$\beta_k=B=0$$ in \eqref{key}. Recall that then we have 
$$A\in i\RR{}$$ and 
$$\varphi_s(z,\bar z,u)=|z|^s$$ in \eqref{mnonminimal}.  According to \autoref{exceptcase}, we can (formally) normalize the vector field as 
\begin{equation}\label{nfspec2a}
 X =   i z w^{k}(1+\sum_{j\geq 1} c_jw^j)  \dop{z} 
+  r w^{k+1}(w^q+...) \dop{w} 
\end{equation}
where $q \geq 1$ and dots stand for terms of higher order. This is what one gets from the normalization procedure in Section 2, {\em and to normalize further one needs to consider the further nondegenerate terms in \eqref{nfspec2a}}. 

Assume first $r\neq 0$. We claim that $X$, as in  \eqref{nfspec2a}, can be then further normalized as
\begin{equation}\label{nfspec3}
 X =   i z w^{k}(1+\tilde c_1w+...+\tilde c_{q}w^{q})  \dop{z} 
+  (r w^{k+1+q}+tw^{2q+2k+1}) \dop{w}.
\end{equation}
Indeed, applying a transformation
$$z\mapsto zf(w), \quad w\mapsto w+wg(w), \quad f(0)=1,\,\,g(0)=0,$$ and working out the desired transformation property, we obtain the following system of ODEs for $f,g$:
\beq\label{DEs}
\begin{aligned}
&i\Bigl(1+\sum_{j\geq 1} c_jw^j\Bigr)f+(rw^{q+1}+...)f'= i(1+g)^k\Bigl(1+\tilde c_1w(1+g)+...+\tilde c_{q}w^{q}(1+g)^{q}\Bigr)f, \\
&r(1+...)(1+wg'+g)=r(1+g)^{q+k+1}+tw^{q+k}(1+g)^{2q+2k+1}.
\end{aligned}
\eeq
The second ODE in \eqref{DEs} can be solved uniquely for $g$ with $g(0)=1$, except in the degree $q+k$ correponding to the resonant term $w^{q+k}$ in $g$; in the latter case, the unique solvability is accomplished by an appropriate choice of the coefficient $t$. We then substitute the determined above function $g(w)$ (which is now fixed) into the first ODE in \eqref{DEs}, and find the unique collection of coefficients $\tilde c_1,...,\tilde c_q$ for which it holds that
$$(1+\sum_{j\geq 1} c_jw^j\Bigr) - (1+g)^k\Bigl(1+\tilde c_1w(1+g)+...+\tilde c_{q}w^{q}(1+g)^{q}\Bigr) = O(w^{q+1}).$$
With the latter property, the first ODE in \eqref{DEs} turns into a regular ODE for $f$ of the kind
$f'=A(w)f$ with a holomorphic $A(w)$, which we solve uniquely with $f(0)=1$. This proves the existence of the further normalization \eqref{nfspec3}.

We shall further take into account the existence of a real integral manifold for $X$.
For a Levi-nonflat hypersurface $v=\psi(z,\bar z,u)$ passing through $0$, we write down the tangency condition with $\re X$ of the form \eqref{nfspec3} and get:

\begin{equation}\label{basic1}
\re \left( r (u+ i \psi)^{k+q+1} +t (u+ i \psi)^{2k+2q+1}  \right) \left(\frac{1}{2i} - \frac12 \psi_u \right) 
= 
  \re\bigl(i z  \psi_z (u+ i \psi)^k (1 + \tilde c_1(u + i \psi) + \dots + \tilde c_q(u + i \psi)^q )\bigr)
\end{equation}

First we observe that, if such $M$ exists, then one necessarily has 
$$r\in\RR{}.$$ Indeed, arguing by contradiction, we see that $r$ with $\im r\neq 0$ would produce in \eqref{basic1} a nonzero term with $u^{k+q+1}$ which is not compensated by any other term, which gives a contradiction. Second, we claim that
$$c_j \in \mathbb R,\quad j=1, \dots, q.$$ Indeed, arguing by contradiction, we take the smallest $j$ with $\Im c_j \neq 0$. It produces in the tangency equation \eqref{basic1} a nonzero term of the form $u^{k+j+m}|z|^s$. On the other hand, there is no other contribution to such a monomial, which leads to contradiction.

We will further show the converse, i.e., for every choice of $c_j \in \mathbb R$, $j = 1, \dots, q$  there is a real integral hypersurface of the vector field \eqref{nfspec3}.

By a similar argument as in \autoref{B=0},  $\psi$ has to be of the form
$$
v = \psi(z,\bar z, u) = \sum_{j=0}^{\infty} \psi_j (u) \vert z\vert^{2j}.
$$
Based on that, we will be looking for an integral manifold with a defining function of the form
$$\psi(z,\bar z,u)=u^{k+q+1}\varphi(|z|^2,u)$$
with an analytic near the origin $\varphi=\varphi(t,u)$.
Plugging such a defining function into the tangency condition of $M$ with a vector field $X$, as in \eqref{nfspec3}, we obtain:
\beq\label{tangnew}
\im (r w^{k+1+q}+tw^{2q+2k+1})=\bigl(u^{k+q+1}\varphi_u+(k+q+1)u^{k+q}\varphi\bigr)\cdot\re(r w^{k+1+q}+tw^{2q+2k+1})+
\eeq
$$+2u^{k+q+1}\re\bigl(iz\bar z\varphi_t\cdot w^{k}(1+\tilde c_1w+...+\tilde c_{q}w^{q})\bigr)\Bigl|_{w=u+iu^{k+q+1}\varphi}.$$ 
The latter identity, after dividing by $u^{2k+2q+1}$ and simplifying, turns into a {\em nonsingular} PDE for  $\varphi=\varphi(t,u)$:
\beq\label{FuchsPDE}
\varphi_u=H(u,t,\varphi,\varphi_t),
\eeq
where $H$ is holomorphic near the origin (this is due to the fact that the linear terms $const\cdot\varphi$ cancel in \eqref{tangnew}). This proves that for any analytic initial data $C(t)=\varphi(0,t)$ there exists a (unique) analytic solution of the respective Cauchy-Kovalewski problem, giving the desired integral manifolds.

Hence we obtain

\begin{theorem}\label{nfspec2}
Every vector field $X$ with $r\neq 0$ in \eqref{nfspec2a} can be brought by a formal invertible  transformation to the normal form

\begin{equation}\label{nfspec3a}
 X =   i z w^{k}(1+ c_1w+...+ c_{q}w^{q})  \dop{z} 
+  (r w^{k+1+q}+tw^{2q+2k+1}) \dop{w}, 
\end{equation}
where $k \ge 1, q \ge 1,  r \in \mathbb R^*, t, c_1, \dots, c_q \in \mathbb R$. 
All the parameters present in the normal form \eqref{nfspec1} actually occur, i.e. for each collection of $k\geq 1, q \ge 1, r\in \mathbb R^*, t, c_1, \dots, c_q  \in \mathbb R $ there exists a formal Levi-nonflat hypersurface with $X\in\mathfrak{aut}\,(M,0)$. 
\end{theorem}

Now consider the case $r = 0$.
We claim that $k=0$ in this case. Indeed, for $k>0$, we first easily deduce that $\psi=\psi(|z|^2,u)$ (this is obtained by considering in \eqref{basic1} one by one terms of minimal total degree in $z,\bar z$). Then, considering in \eqref{basic1} the term $z^i\bar z^ju^l$ with minimal possible $i+j$, then minimal possible $l$ and then minimal possible $i$ generates in \eqref{basic1} a nonzero term with $z^{2i}\bar z^{2j}u^{2l}$, which on the other hand must vanish, which is a contradiction. Thus, we conclude that
  $k=0$ and 
\begin{equation}\label{except}
X=iz\dop{z}.
\end{equation}

Thus we arrive at 

\begin{proposition}\label{nfspec1f}
Every vector field $X$ with $\alpha_k\equiv 0$ in \eqref{key} and $r=0$ in \eqref{nfspec2a}  which is tangent to a real-analytic Levi-nonflat hypersurface satisfies $k=0$ and can be brought by a formal invertible  transformation to the normal form \eqref{except}.
 
\end{proposition}

Our immediate goal is to investigate the convergence of the normal forms \eqref{nfspec3}, \eqref{except}. 

In the case of rotational normal form \eqref{except}, the problem is treated identically to the above exceptional case $A=0,B\neq 0$, and we obtain
\begin{proposition}
In the assumptions of \autoref{nfspec1}, there is a biholomorphic transformation bringing the vector field $X$ to the normal form \eqref{except}. 
\end{proposition}

In the case of normal form \eqref{nfspec3}, in contrast, we face the {\em divergence phenomenon}. To show this, we first need
\begin{lemma}\label{centralizer}
The space of formal vector fields $Y$ at the origin commuting with a vector field $X$ in normal form \eqref{nfspec3} with $r\neq 0$ (the {\em formal centralizer} of $X$) is spanned by $X$ and the rotation field \eqref{except}.

At the same time, for $r=0$ (and hence $t=0$) the formal centralizer is infinite-dimensional.
\end{lemma}
\begin{proof}
Assume first $r\neq 0$ in \eqref{nfspec3}. Take a formal vector field $Y=f\dz+g\dw$ and consider the identity $[X.Y]=0$. Calculating the coefficients for $\dz$ and $\dw$ respectively, we obtain:
\beq\label{comm}
 i z w^{k}(1+\tilde c_1w+...+\tilde c_{q}w^{q})f_z+(r w^{k+1+q}+tw^{2q+2k+1}) f_w-ifw^{k}(1+\tilde c_1w+...+\tilde c_{q}w^{q})-izg(w^k+\tilde c_1w^{k+1}+...+\tilde c_{q}w^{k+q})'=0,
\eeq
$$ i z w^{k}(1+\tilde c_1w+...+\tilde c_{q}w^{q})g_z+(r w^{k+1+q}+tw^{2q+2k+1}) g_w-(r w^{k+1+q}+tw^{2q+2k+1})'g=0. $$
We first deal with the second equation in \eqref{comm}. Expanding 
$$g=\sum_{j\geq 0} g_j(w)z^j,$$
plugging into the second equation in \eqref{comm} and gathering terms with $z^j,\,j\geq 0$, we obtain:
$$\bigl(ijw^k(1+\tilde c_1w+...+\tilde c_{q}w^{q})-(r w^{k+1+q}+tw^{2q+2k+1})'\bigr)g_j+(r w^{k+1+q}+tw^{2q+2k+1})g_j'=0.$$
Note that $g_j'/g_j$ must have a pole of order $1$ at $0$, unless $g_j\equiv 0$. Combining this observation with the latter ODE for $g_j$, we see that $g_j\equiv 0$ for $j\neq 0$. For $j=0$, we get  
$$g_0=C(r w^{k+1+q}+tw^{2q+2k+1}),\quad C\in\CC{}.$$
Plugging the resulting $g=g_0$ into the first equation in \eqref{comm}, expanding
$$f=\sum_{j\geq 0} f_j(w)z^j,$$
and gathering terms with $z^j$, we similarly conclude that $f_j\equiv 0$ for $j\neq 1$, while for $f_1$ we have 
$$f_1=Cw^{k}(1+\tilde c_1w+...+\tilde c_{q}w^{q})+D,\quad C,D\in\CC{},$$
as required.

In the case $r=0$, we observe that the centralizer contains all formal vector fields of the kind $f(w)z\dz$, which shows the infinite-dimensionality and 
proves the lemma.

\end{proof}

We are now in the position to prove
\begin{theorem}\label{diverg}
For all the vector fields
\beq\label{e:diverg}
X_k:=(w^{k+1}+izw^k)\dz+w^{k+2}\dw,\quad k>0,
\eeq
any transformation bringing $X_k$ to the normal form \eqref{nfspec3} is divergent.
\end{theorem}
\begin{proof}

We start by investigating the formal centralizer of $X_k$. For doing so, we study the condition $[X_k,Y]=0$ for a formal vector field $Y=f\dz+g\dw$ and collect respectively the $\dz,\dw$ components. This gives:
\beq\label{divsyst}
(w^{k+1}+izw^k)f_z+w^{k+2}f_w-iw^kf-((k+1)w^k+ikzw^{k-1})g=0,
\eeq
$$(w^{k+1}+izw^k)g_z+w^{k+2}g_w-(k+2)w^{k+1}g=0.$$
We first show that the space of formal solutions of \eqref{divsyst} is finite-dimensional. For doing so, we start with the second equation. Expanding 
$$g=\sum_{j\geq 0}g_j(w)z^j,\quad f=\sum_{j\geq 0}f_j(w)z^j$$
and gathering terms with $z^j$, we get:
$$ijw^kg_{j}+(j+1)w^{k+1}g_{j+1}+w^{k+2}g_j'-(k+2)w^{k+1}g_j=0.$$
For $j\neq 0$, we first conclude from the latter identity that 
$$\mbox{ord}_0\,g_{j+1}=\mbox{ord}_0\,g_{j}-1,$$ 
and this proves that {\em all $g_j$ vanish for $j> N$ for some $N$.} The remaining coefficient functions $g_N,\dots,g_0$ successively satisfy  ODEs of the kind
$$w^2g_j' +(ij-(k+2)w)g_j=H_j(w),$$
where $H_j(w)$ is a formal power series (obtained by substituting the previously determined $g_l,\,l>j$). This shows that the space of formal solutions $\{g\}$ for the second PDE in \eqref{divsyst} is finite-dimensional. 

Taking the property $g_j=0$ for $j>N$ into account and using a similar argument, we conclude a similar property of  $f_j$ (with a possibly different $N$) and an analogous finite-dimensionality property. This proves that the formal centralizer of $X_k$ is finite-dimensional. 

We arrive to the conclusion, in view of \autoref{centralizer}, that {\em the normal form \eqref{nfspec3} has $r\neq 0$}.

We proceed with the proof of divergence now. In view of \autoref{centralizer}, $X_k$ admits a (formal) vector field $Y$ commuting with $X_k$ with 
\beq\label{theY}
Y=f\dz+g\dw=(iz+O(w))\dz+O(w^2)\dw 
\eeq
(it is obtained by pulling back the rotation field \eqref{except} by the normalizing transformation).  Then $f,g$ satisfy, again, the system \eqref{divsyst}.
Let's search for solutions of \eqref{divsyst} of the kind:
$$g=0, \quad f=f_0(w)+cz,\quad c\in\CC{}.$$
For such a couple $(f,g)$ being a solution amounts to the single ODE:
$$
w^2f_0'-if_0=-cw.
$$
Expanding 
$$f_0=\sum_{l\geq 1} a_lw^l$$
and collecting terms of fixed degree, we first obtain $a_1=-ic$, and then $a_{l+1}=-i(l+1)a_l,\,l\geq 1$. This proves that $f_0$ is a {\em divergent} formal power series for $c\neq 0$. The latter means that there exists no convergent transformation of $X_k$ to the normal form \eqref{nfspec3} (since in the normal form the entire centralizer is analytic, as follows from \autoref{centralizer}). This completely proves the theorem. 

\end{proof}

\end{document}